\documentclass[12pt]{amsart}
\usepackage{amscd}
\usepackage[all]{xy}
\usepackage[notcite,notref]{showkeys}

\newcommand{\rank}{{\operatorname{rank}}}
\newcommand{\supp}{{\operatorname{supp}}}

\newcommand{\CH}{\operatorname{CH}}

\newcommand{\Aut}{\operatorname{Aut}}

\newcommand{\Span}{\operatorname{Span}}

\newcommand{\C}{\mathbf{C}}

\renewcommand{\P}{\mathbf{P}}

\newcommand{\Z}{\mathbf{Z}}

 \newcommand{\sC}{\mathcal{C}}
\newcommand{\sG}{\mathcal{G}}
\newcommand{\sM}{\mathcal{M}}

\newcommand{\sO}{\mathcal{O}}

\newcommand{\sT}{\mathcal{T}}

\newcommand{\sS}{\mathcal{S}}
 
 \newcommand{\sF}{\mathcal{F}}

\newcommand{\sA}{\mathcal{A}}

 \numberwithin{equation}{section}

\theoremstyle{plain}
\newtheorem{thm}[equation]{Theorem}
\newtheorem{prop}[equation]{Proposition}

\theoremstyle{definition}

\newtheorem{ex}[equation]{Example}

\begin{document}
 \title{Rational cubic fourfolds with a symplectic group of automorphisms }
 \author[C. Pedrini]{Claudio Pedrini}
\address{Dipartimento di Matematica \\ %
Universit\'a degli Studi di Genova \\ %
Via Dodecaneso 35 \\ %
16146 Genova \\ %
Italy}
\email{claudiopedrini4@gmail.com}

\begin{abstract}A well known conjecture asserts that a cubic fourfold $X$ is rational if  it has a cohomologically associated K3 surface. G.Ouchi in [Ou] proved that if $X$ admits a  finite group $G$
of symplectic automorphisms, whose order is different from 2, then $X$ has an associated K3 surface $S$ in the derived sense, i.e. $\sA_X \simeq D^b(S) $, where  $\sA_X$ is the Kuznetsov component  of the derived category $D^b(X)$.This is equivalent to have a cohomologically associated K3 surface  and therefore  $X$ is conjecturally rational. In  this note we prove that  cubic fourfolds with
a cyclic group $G$ of symplectic automorphisms , whose order is not a power of 2,  are rational and belong to the Hassett divisor $\sC_d$, with $d = 14 ,42$. We also describe rational cubic fourfolds $ X$
with a symplectic group of automorphisms $G$, such that $(G, S_G(X)$, is a {\it Lech pair} with $\rank S_G(X) =19,20$.

\end{abstract}
\maketitle 

\section{Introduction}
Let $\sC$ be the moduli space of smooth cubic fourfolds. A well known conjecture asserts that a cubic fourfold  $X$ is rational if  it has  a cohomologically associated  K3 surface $S$, meaning that the 
middle cohomology of $X$ -as a Hodge structure- contains the primitive cohomology of $S$. This is an explicit countable union of divisors $\sC_d$ in $\sC$, where $d$ satisfies the following numerical condition\par
(*) $d>6$ and  $d$ is not divisible by 4,9 or a prime $p \equiv 2 (3)$\par
In [BLMNPS, Cor.29.7]  it has been proved that a smooth cubic fourfold has a cohomologically associated K3 surface $S$ if and only if there is an equivalence of categories  $\sA_X \simeq D^b(S)$. Here
$\sA_X$ is the Kuznetsov  component of the derived category $D^b(X),$ i.e
$$D^b(X) = <\sA_X, \sO_X, \sO_X(1),\sO_X(2)>$$
Cubics in $\sC_{14} , \sC_{26}, \sC_{38},\sC_{42}$ are known to be rational, see [RS]. In all these cases the  construction of a birational map $\mu : X \dashrightarrow W$, where $W$ is either $\P^4$ or a notable smooth rational  fourfold, can be  achieved by considering a surface $S_d \subset X$ admitting a four dimensional family of $(3e-1)$-secant curves of degree $e\ge 2$. These curves are parametrized by a rational variety with the property that through a general point of $\P^5$ there passes a unique curve of the family,  The inverse $\mu^{-1}: W \dashrightarrow X$ is given by a linear system of divisors in $\vert \sO_W(e\cdot i(W) -1)\vert$, where $i(W)$ is the index of $W$, having   points of multiplicity $e$ along  an irreducible surface $U \subset W$
 which is a birational incarnation of a K3 surface $S$ associated to $X$.\par
 The divisors $\sC_8$ and $\sC_{18}$ contain countably many rational cubic fourfolds. Recently B.Hassett , in [Hass 2], showed that
there is a countable union of codimension two loci in $\sC_{24} $ containing rational cubic fourfolds. The intersection $\sC_8 \cap \sC_{12} $ contains 3 irreducible components one of which consists
of rational cubic fourfolds, see [BBH]. \par
A different approach in order to describe familes of rational cubic fourfolds  $X$ is to consider  the case when $X$ admits a finite group of   automorphisms.\par
The following result has been proved by G.Ouchi in [Ou,Thm.8.2] 
\begin{thm}\label{Ou}  Let $G$ be a group of symplectic automorphisms of a smooth cubic fourfold $X$. If the order of  $G$   is different from 2  there exists a unique K3 surface $S$ such that $\sA_X \simeq D^b(S) $\end{thm}
A cubic fourfold as in \ref{Ou} is conjecturally rational and therefore should belong to a divisor $\sC_d$ with $d$ as in (*).  Cubic fourfolds with a symplectic involution have no associated K3 surfaces and  are conjecturally irrational. On the other hand  there are cubics with a non-symplectic  automorphism of order 3 that are known to be rational see [Ped].\par
The group of regular automorphisms of a cubic fourfold  $X\subset \P^5$ is  finite and any automorphism $f $  of $X$ is induced from a linear automorphism of the ambient projective space. Therefore the induced automorphism  $f^*$ on $ H^4(X,\Z)$ preserves $h^2 $, where $h$ is the class of a hyperplane section. The Abel-Jacobi map 
 $$p_*q^* : H^4(X,\Z) \simeq H^2(F(X),\Z),$$
 \noindent where $P_F \subset F(X) \times X$ is the universal family and $p,q$ are the projections to $F(X)$ and $X$, respectively, induces an isomorphism of Hodge structures 
 $$H^{3,1}(X) \simeq H^{2,0}(F(X)).$$
Let  $G$ be  a  finite group isomorphic to a subgroup of $\Aut(X)$, with $X$ a smooth cubic fourfold. Then 
$$\vert G \vert=2^{a_2} 3^{a_3} 5^{a_5} 7^{a_7}11^{a_{11}},$$
\noindent with $a_2 \le 5,a_3\le 7,a_5 \le 1, a_{11} \le 1$, see [YYZ,(6.1)].\par
An automorphism $f$ of a cubic fourfold $X$  $f$ is symplectic if $f^*$ acts as the identity on $H^{2,0}(F(X)$, that is equivalent to $f^*$ acting  as the identity  on $H^{3,1}(X) \simeq \C\omega$. \par 
If  $G = <\sigma> $ is a cyclic group acting symplectically  on a smooth cubic fourfold $X$ the order $n$  of $G$ is one of the following 
$$\{2,3,4,5,6,7,8,9,11,12,15\},$$
\noindent  see [LZ,Thm.4.17].\par
In [BGM] it is proved that, for $n =3,5,7,11$ the fourfold $X$ is rational. In Sect.2   we prove  that   for $n = 6,9,12,15$   the fourfold $X$ is rational and belongs either to $\sC_{14}$ or to $\sC_{42}$.In the case
$n=4,8$ there is a subfamily of rational cubics belonging to $\sC_8 \cap \sC_{12}$.\par
 In Sect.3 we describe  rational  cubic fourfolds $X$ such that $(X,S_G(X))$ is a {\it Lech pair}, with $ \rank S_G(X)=19, 20$,. Here $S_G(X)$ is the coinvariant lattice, i.e. the orthogonal complement of the 
invariant sublattice of $H^4(X,\Z)$ under the induced action of $G$.\par
In Sect.4 we consider cubic fourfolds $X$ with a symplectic automorphisms  of prime order and show when the automorphism induced on $F(X)$ is natural, i.e. comes from a symplectic automorphism of the associated K3 surface, see \ref{natural}.  In \ref{equivariant}  we consider the {\it equivariant  Kuznetsov component} $\sA^G_X$  and we prove that, if $X \in \sC_{6d}$ ( for  some suitable  $d$) has a order 3 group $G$ of symplectic automorphisms , then there exists a cubic fourfold  $X' \in \sC_{2d} $ such that $\sA^G_X \simeq \sA_{X'}$.

 \section { Cyclic groups of  symplectic automorphisms} 
In this section we prove that every cubic   fourfold admitting a cyclic group of symplectic automorphisms, whose order $n$ is not a power of 2, is rational and belongs  either to $\sC_{14}$ or to $\sC_{42}$. We also show that, in the case $n=4,8$ ,the family of cubic fourfolds with a cyclyc group of order $n$ contains a subfamily of rational cubics belonging to $\sC_8 \cap \C_{12}.$\par
Symplectic automorphisms of order 3  acting on a smooth cubic fourfold $X \subset \P^5_{\C}$ has been classified   in [Fu] and [GAL].\par
Let $\sigma_1,\sigma_2,\sigma_3 $ be the following automorphisms of $\P^5_{\C}$ of order 3
$$ \sigma_1 : [x_0,x_1,x_2,x_3,x_4,x_5]  \to :[x_0,x_1,x_2 ,x_3,\zeta x_4, \zeta^2 x_5] ;$$
$$ \sigma_2:  [x_0,x_1,x_2,x _3,x_4,x_5]  \to :[x_0,x_1,x_2 , \zeta x_3,\zeta x_4, \zeta x_5]  ;$$
$$\sigma_3: [x_0,x_1,x_2,x_3,x_4,x_5] \to :[x_0,x_1,\zeta x_2,\zeta x_3,\zeta^2x_4, \zeta^2 x_5] .$$
\noindent with $\zeta^3=1$.  Let $V_i$, for $i=1,2,3$, be the families  of cubic fourfolds, invariant  under $\sigma_i$ on which  the automorphism acts symplectically , see [GAL,Thm.3.8] . \par
Let $V_1$ be the family of cubic fourfolds invariant under the automorphism $\sigma_1$.  Every $X \in V_1$ is defined by an equation of the form 
\begin{equation} \label{V_1}  f(x_0,x_1,x_2,x_3) +x^3_4 +x^3_5 + x_4x_5 l(x_0,x_1,x_2,x_3) =0 \end{equation}
\noindent with $f$ of degree 3 and $l$ of degree 1. \par
The family $V_1$ has dimension 8. The rank of the algebraic lattice
$$A(X) = H^4(X,\Z) \cap H^{2,2}(X)$$ 
\noindent is 13 . The cubic fourfold $X$  contains two disjoint planes, see [BGM, Prop.5.2]. Therefore  is rational and belongs to $\sC_{14}$. The fixed locus of $\sigma_1$ on $F(X)$ consists of  the 27 lines on the cubic surface 
$$ S :  f(x_0,x_1,x_2,x_3) =0 .$$
Let $V_2$ be the family of cubic fourfolds invariant under the automorphism $\sigma_2$. The family $V_2$ has dimension  2.  A  cubic fourfold $X \in V_2$ has an equation of the form
\begin{equation} \label{sum} F(x_0,x_1,x_2,x_3,x_4,x_5) = f(x_0,x_1,x_2,)+g(x_3,x_4,x_5)=0,\end{equation}
\noindent where $f$ and $g$ are homogeneous of degree 3.  
\begin{prop}\label{V_2} Every cubic fourfold in $V_2$ has an associated K3 surface $S$ , is rational and belongs to $\sC_{14}$. The fixed locus of $\sigma_2$ on $F(X)$ is 
isomorphic to the product $C \times D$, with $C,D$ elliptic curves. A minimal resolution of the quotient $F(X) /\sigma^*_2$  is  a hyper-K\"alher manifold birational to the generalized Kummer manifold $K^2(C \times D)$,\end{prop}
\begin{proof}  The fixed locus of the automorphism $\sigma_2$ on $X$ is the union of the cubic curve $C : f(x_0,x_1,x_2) =0$ contained in $\Pi_1: (x_3=x_4=x_5 =0)$  and the cubic  curve 
$D : f(x_3,x_4,x_5) =0$ contained in $\Pi_2 :(x_0=x_1=x_2=0)$. The fixed locus of $\sigma^*_2$ on $F(X)$ is the locus of lines joining a point $Q_1$ on the  curve $C $  and a  point $Q_2 $ on the   curve $D $, see [Fu,Thm 1]. These lines are parametrized by  $C \times D$ which is an abelian surface.\par
The quotient  $F(X)/\sigma^*_2$ has an $A_2$-singularity along the abelian variety  $C \times D$ and has a minimal resolution $Y$ which is birational to the generalized Kummer  manifold $K^2(C \times D)$,
see [Kaw, 3.2].\par
The fourfold $X$ contains 2 disjoint planes $P_1, P_2$, see [CT, Rk.2.4]. Therefore there is a rational dominant map
$$P_1 \times P_2 \simeq \P^2 \times \P^2 \dashrightarrow X$$
\noindent whose indeterminacy locus consists of the lines joining a point $p_1 \in P_1$ and $p_2 \in P_2$  which are contained in $X$. This locus is parametrized by a K3 surface $S$, which is associated to $X$. Therefore $\sA_X \simeq D^b(S)$.\par
 \end{proof} 
 The family  $V_3$ of cubic fourfolds $X$ that are invariant under the automorphism  
$$\sigma_3 :[x_0,x_1,x_2,x_3,x_4,x_5] \to :[x_0,x_1,\zeta x_2,\zeta x_3,\zeta^2x_4, \zeta^2 x_5],$$
\noindent acting symplectically on $F(X)$  consists of cubics defined by an equation of the form
\begin{equation}\label{V_3}  f(x_1,x_2) +g(x_2,x_3) +h(x_4, x_5) + \sum_{i,j,k}a_{ijk} x_i ix_j x_k =0,\end{equation}
\noindent with $f,g,h$ homogeneous of degree3 and $i=0,1 ; j =2,3.;  k=4,5$, see [Fu,Thm.1.1].\par
The family  $V_3$ has dimension 8. The algebraic lattice $A(X)$ has rank 13 .The fixed locus of $\sigma_3$ on $X$ consists of 9 isolated points and the fixed locus of $\sigma^*_3$ on $F(X)$ of 27 lines such that each fixed point on $X$ is the intersection of 3 lines in $F(X)$, see [Fu, IV-(5)].  There is a primitive sublattive of $A(X)_{prim}$ with Gram matrix
$$  
\begin{pmatrix} 
4&1&0 \\
1 &4&0\\
0&0&4 \\
\end{pmatrix}
$$
\noindent  Therefore the  algebraic lattice $A(X)$ contains a lattice $K =<h^2,v>$ with $v \in A(X)_{prim}$ and $v^2 =14$, see [BGM,Cor.3.5] .The cubic fourfold  $X $  belongs to $ \sC_{42} $ and is rational.

\subsection{Automorphisms of order 5,7,11}
Let  $F_5$  be the family of cubic fourfolds  $X \subset \P^5$ which are  invariant under the symplectic automorphism $\phi_5$ of order 5 of $\P^5$
\begin{equation}    \phi_5  : [x_0,x_1,x_2,x_3,x_4,x_5]  \to :[x_0,x_1, \zeta x_2 , \zeta^2 x_3, \zeta^3 x_4 ,  \zeta^4 x_5],\end{equation}
\noindent where $\zeta^5 =1$.The family $F_5$ has dimension 4.  Every $X \in F_5$ has an equation of the form
\begin{equation}  \label{equation} F(x_0,x_1,x_2,x_3,x_4,x_5)=f(x_0,x_1) +x_2x_5 l_1(x_0,x_1) +x_3 x_4 l_2(x_0,x_1) +  x^2_2x_4+x_2x^2_3 +x_3x^2_5 +x^2_4x_5,\end{equation}
with $f$ of degree 3 and $l_1,l_2$  of degree 1, see  [GAL,Thm.3.8].  A general cubic fourfold $X \in F_5$ has a symplectic group of isomorphisms $G =D_{10}$  generated by the automorphism $\phi_5$ and 
by the symplectic involution
$$\tau :  [x_0,x_1,x_2,x_3,x_4,x_5]   \to [x_0,x_1,x_5,x_4,x_3,x_2],$$
\noindent  see [LZ,Thm.1.2 (5)].\par
The fixed locus of  $\phi^*_5$ on on $F(X) $ consists of 14 isolated points, see [Fu, Thm. 1.1]. The algebraic  lattice $A(X)$  has order 17. 
 A cubic fourfold  $X \in F_5$  is rational and belongs to $\sC_{42}$, see [BGM. Cor.3.5].\par
Let  $F_7$  be the family of cubic fourfolds $X$ invariant under the symplectic automorphism $\phi_7$ of order 7 of $\P^5$
$$\phi_7 :  [x_0,x_1,x_2,x_3,x_4,x_5]  \to :[ \zeta x_0,\zeta^5 x_1, \zeta ^4 x_2 , \zeta^6 x_3, \zeta^2 x_4 ,  \zeta^3 x_5],$$
\noindent  where $\zeta^7 =1$. The family $F_7$  has dimension 2. Every $X \in F_7$ has an equation of the form
$$F(x_0,x_1,x_2,x_3,x_4,x_5) = x^2_0 x_4+x^2_1 x_2+x_0 x^2_2+x^2_3 x_5+x_3 x^2_4+x_1 x^2_5+ax_0 x_1 x_3+b x_2x_4 x_5=0$$
The fixed locus of $\phi_7$ on $F(X)$ consists of 9 isolated points. The algebraic lattice$A(X)$ splits as $<h^2> \oplus A(X)_{prim} $  with $A(X)_{prin} $ of rank 18.\par
A general cubic $X \in F_7$ has a symplectic group $G$ of order 21 generated by the automorphism $\phi_7 $ of order7
and by the order  3 symplectic automorphism 
$$\tau :  [x_0,x_1,x_2,x_3,x_4,x_5]  \to [x_2,x_0,x_1,x_4,x_5,x_3]. $$
The family $F_7$ is contained in $V_3$, hence every $X \in \sF_7$ is rational and belongs to $\sC_{42}$.\par
Let  $F_{11}$ be the family of cubic fourfolds invariant under the symplectic automorphism
 $$\phi_{11} :  [x_0,x_1,x_2,x_3,x_4,x_5] \to[x_0, \zeta x_1, \zeta^3 x_2, \zeta^4 x_3, \zeta^5 x_4, \zeta^9 x_5]$$
\noindent where $\zeta^{11} =1$. The family $F_{11}$ has dimension 0 and consists of the Klein cubic fourfold $X_K$, the only cubic fourfold with a symplectic automorphism of order 11. $X_K$ is the triple cover of $\P^4$ ramified on the Klein cubic threefold, defined by the equation
\begin{equation} \label{Klein}  x^2_0x_1+x^2_1x_2+x^2_2 x_3+x^2_3x_4 +x^2_4x_0+x^3_5 =0.\end{equation} 
The cubic $X_K$ has a group $G$  of symplectic  automorphisms isomorphic to $L_2(11)$.The group  $G$ is generated by   $\phi_{11}$ and  by  the symplectic automorphism of order 5
 $$ \tau :   [x_0,x_1,x_2,x_3,x_4,x_5] \to [x_0,x_3,x_2,x_4,x_5, x_2].$$ 
 Therefore $X_K \in F_5 $, is rational and  belongs to $\sC_{42} $. The fixed locus of  $\phi_{11}$ on on $F(X) $ consists of 5 isolated points, see [Fu,Thm.1.1] .\par
 
 \subsection{Cyclic groups of order order 6 and 9}
There are two families $\sF_1 , \sF_2$, each one of dimension 4 , of cubic fourfolds $X$ with a symplectic automorphism $ \sigma$ of order 6, see [LZ, Lemma 4.22]. For an appropriate choice of coordinates  the defining equation for $X \in \sF_1 $ belongs to

$$ \Span \{ x^2_0x_2,x^2_0x_3,x_0 x _1 x_2,x_0 x_1x_3,x^2_1x_2, x^2_1x_3,x^3_2,x^2_2x_3,x_2 x^2_3,x_2x_4x_5,x^3_3,x_3x_4x_5, x^3_4,x^3_5\}$$

\noindent and the automorphism $\sigma$ acts as follows
$$\sigma :   [x_0,x_1,x_2,x_3, x_4,x_5] \to [\zeta^3 x_0,\zeta^3 x_1, x_2,x_3, \zeta^2 x_4, \zeta ^2x_5],$$
\noindent with $\zeta$  a  primitive $6^{th}$- root of 1. The  order 3 automorphism $\sigma^2$ acts as follows
$$\sigma^2 :   [x_0,x_1,x_2,x_3, x_4,x_5]  \to   [x_0,x_1,x_2,x_3,  \omega x_4, \omega^2 x_5],$$ 
\noindent with $\omega$ a primitive third root of 1, see [LZ, Lemma 4.22].  Therefore  $X \in V_1$  and it is defined by an equation of the form as in \ref{V_1}
 The cubic fourfold $X$ contains two disjoint planes, is rational and belongs to $\sC_{14}$.\par
For an appropriate choice of coordinates  the defining equation for $X \in \sF_2 $ belongs to
$$ \Span \{ x^3_0, x_0 x^2_1,x_0 x_2 x_4, x_0 x_2 x_5, x_1 x_3 x_4, x_1 x_3 x_5 ,x^3_2, x_2x^2_3, x^3_4, x^2_4x_5, x_4 x^2_5,x^3_5\}$$
\noindent and the  order 3 automorphism $\sigma^2$ acts as follows
\begin{equation} \label{square}\sigma^2 :   [x_0,x_1,x_2,x_3, x_4,x_5]  \to   [x_0,x_1,\omega x_2,\omega x_3,  \omega^2 x_4, \omega^2 x_5], \end{equation}
\noindent with $\omega$ a primitive third root of 1.Therefore $X \in V_3$ and its equation is of the form  as in \ref{V_3} .
The fourfold  $X$  is rational and belongs to $\sC_{42}$.\par
  If  $G$ is  a  cyclic group $G$ of order 9 , acting symplectically on a cubic fourfold $X$ we can choose  coordinates $[x_0,x_1,x_2,x_3, x_4,x_5]  $ on $\P^5$  and a generator $g \in G$ such that the automorphism induced  by $g$ on $X$ is  one of the following
$$ g_1 : [x_0,x_1,x_2,x_3, x_4,x_5]  \to [x_0,\zeta^6 x_1, \zeta ^3 x_2, \zeta x_3, \zeta 4 x_4,\zeta^7 x_5] $$
$$ g_2 :  [x_0,x_1,x_2,x_3, x_4,x_5] \to [x_0,\zeta^3 x_1, \zeta^6 x_2,\zeta x_3, \zeta x_4,\zeta^4 x_5],$$
\noindent where $\zeta $ is a primitive  $9^{th}$-root of 1. In the first case  $X$ has an equation $F(x_0,x_1,x_2,x_3, x_4,x_5)=0$ such that

$$ F \in \Span\{ x^2_0 x_1,x^2_1 x_2,x^2_2 x_0,x^2_3 x_4,x^2_4 x_5,x^2_5 x_3\}$$
\noindent  and in the second case

$$ F \in \Span\{ x^2_0 x_1, x^2_1 x_2, x^2_2 x_0, x^2_3 x_4, x_3 x^2_4, x^3_3, x^3_4, x^3_5\}$$
\noindent see [LZ, Thm. 4.15]. \par
In both cases the equation $F(x_0,x_1,x_2,x_3, x_4,x_5)=0 $ defining $X$ is of  the form as in \ref{V_2} 
$$F(x_0,x_1,x_2,x_3, x_4,x_5)= f(x_0,x_1,x_2) +g(x_3,x_4,x_5)=0$$
\noindent with $f,g$ of degree 3,. Therefore  $X \in V_2$, is rational and belongs to $\sC_{14}$
\subsection{Cyclic groups of orders 12 and 15} In [LZ, Thm.1.8] a list is given of all the groups $G$ which are the symplectic automorphism 
group of a cubic fourfold $X$ with rank $S_G(X) =20$. Here $S_G(X)$ is the coinvariant lattice, i.e. the orthogonal complement of the 
invariant sublattice of $H^4(X,\Z)$ under the induced action of $G$. The associated   cubic fourfolds ( automatically isolated in moduli) includes the following ones, with an automorphism of order  12,15 .\par
(1) the cubic fourfold  $X_{12}$ defined by the equation :
\begin{equation} \label{12} F(x_0,x_1,x_2,x_3,x_4,x_5) = x^3_0+x^3 _1+x^3_2+x^3_3+x^3_4+x^3_5 -3(\sqrt{3} +1) (x_0 x_1x_2+ x_3x_4x_5)) =0.\end{equation}
This is the only cubic fourfold with a symplectic automorphism of order 12. It is rational because it belongs to the family $V_2$ in \ref{V_2}  of cubics whose equations are of the form $f(x_0,x_1,x_2,x_3) + g(x_3,x_4,x_5)=0$ with $f$ and $g$  homogeneous of degree 3.Therefore $X$ belongs to $\sC_{14}$.\par
(4) The cubic fourfold  $X_{15}\subset \P^7_{\C} $ defined by the equation 
\begin{equation} \label{15}  x^3_0+x^3 _1+x^3_2+x^3_3+x^3_4+ x^3_5+x^3_6 +x^3_7 = x_0+x_1+ x_2=x_3+x_4+x_5+x_6 +x_7 =0.\end{equation}
This is the only smooth cubic fourfold with a symplectic automorphism  $\sigma$ of order 15. Moreover
$$\Aut(X_{15} )/  \Aut^s(X_{15})\simeq \Z/6\Z,$$
\noindent see [LZ,Thm.1.8(6)].Let 
$$H_1 : x_0+x_1+ x_2=0  \   ,  \  H_2 : x_3+x_4+x_5+x_6 +x_7 =0 ,$$
\noindent be the hyperplanes in $\P^7$. Then $X_{15}$ is defined in $H_1 \cap H_2  =\P^5 \subset \P^7$, with coordinates $[x_1,x_2,x_3,x_4,x_5,x_6]$, by an equation of the form
$$F(x_1,x_2,x_3,x_4,x_5,x_6)= (-x_1-x_2)^3 +x^3_1+x^3_2+   f (x_3,x_4,x_5,x_6) =0,$$
\noindent with $f$ homogeneous of degree 3. The cubic fourfold $X_{15} \subset \P^5$ is invariant under the non-symplectic automorphism 
$$\tau :[x_1,x_2,x_3,x_4,x_5,x_6] \to  [\zeta x_1, \zeta x_2, x_3,x_4,x_5,x_6]$$
\noindent with  $\zeta^3 =1$.  Therefore $X_{15}$ contains two disjoint planes and belongs to $\sC_{14}$, see [Ped , Prop.3.6].\par
\subsection{Cyclic groups of order 4 and 8} Finally we consider the case of the families $\sG_4$ and $\sG_8$ of cubic fourfolds admitting a cyclic group of symplectic automorphisms of order 4 and 8, respectively.
We show that in both cases $\sG_4$ and $\sG_8$ contain subfamilies of cubics which are rational and belong to $\sC_8 \cap \sC_{12}$.\par
A cubic fourfold $X$ with a  cyclic group $ G$ of symplectic automorphisms of order 4 has an equation of the form 
\begin{equation} F(x_0,x_1,x_2,x_3,x_4,x_5) = x_0 N_1(x_2,x_3) +x_1 N_2(x_2,x_3) +F(x_0,x_1) + $$
$$+ x_4x _5 L_1(x_0,x_1) +x^2_4 L_2(x_2,x_3) +x^2_5 L_3(x_2,x_3)=0,\end{equation}
 \noindent with $N_1$,$N_2$ of degree 2, $F$ of degree 3 and $L_1,L_2,L_3$ of degree 1, see [LZ, Thm.4.15]. The group $G$ is generated by 
 $$\sigma : [x_0,x_1,x_2,x_3,x_4,x_5] \to [x_0,x_1,-x_2, - x_3, i x_4, -i x_5].$$
The family  $\sG_4$ of  these cubic fourfolds  has dimension 6.\par
The fixed point set of $\sigma$, viewed as an automorphism of $\P^5$, consists of the disjoint union of the lines  $l_1 =\overline{ P_0 P_1}$, the  line $l_2 =  \overline {P_2 P_3}$ and the points $P_4$,$P_5$. The lines $l_2$ and $l_3= \overline{ P_4 P_5}$ are contained in $X$. The fixed point set of the induced automorphisms $\sigma^*$ on $F(X)$ consists of 16 isolated points, see
[Fu, V-2].\par 
Let $L_2(x_2,x_3) = \alpha x_2+\beta x_3$ , $L_3(x_2,x_3) =\gamma x_2 +\delta x_3$ in the equation defining $X$. A line $l$ joining a point $P= (a_2,a_3) $ on $l_2$  and a point $Q= (b_4,b_5) $ on $l_3$ is contained in $X$ if
the coordinates of the points $P$ and $Q$ satisfy the equation
$$ a_2(\alpha b^2_4 +\gamma b^2_5)+a_3(\beta b^2_4 +\delta b^2_5)=0. $$
This equation defines of rational curve on the surface $l_2 \times l_3 \simeq \P^1 \times \P^1 $ and hence it gives a rational curve $R$ of degree 3 on $F(X)$. To the curve $R$ corresponds a cubic scroll 
$S \subset X$. Therefore $X \in \sC_{12}$. \par
Let $X$ be a cubic fourfold in $\sG_4$ defined by an equation  as in (2.11), with  $L_2(x_2,x_3) = \alpha x_2+\beta x_3$  and $N_1=N_2= (\alpha x_2+\beta x_3)^2$. Then $X$ contains 3 planes $\Pi_i$ of equations
$$\Pi_i  : a_i x_0 + b_i x_1 =\alpha x_2+\beta x_3 =x_5 =0$$
where  $F(a_i, b_i)=0$, for $ i=1,2,3$. Therefore $X \in \sC_8  \cap \sC_{12}$. The intersection $S \cap \Pi_i$ consists of all points on the line $r$ joining 
$P = (0,0,\beta,-\alpha,0,0) \in l_2$ and $P_4 =(0,0,0,0,1,0) \in l_3$. Since $r$ is a line of the ruling of $S$ we have $S \cdot \Pi_i =0$ and therefore $X$ is rational, see [BBH,3.1.3].\par
A cubic fourfold $X\in \sG_8$ is invariant under a symplectic automorphism $\sigma$ of order 8
$$\sigma : [x_0,x_1,x_2,x_3,x_4,x_5] \to [x_0,-x_1,\zeta^2 x_2, \zeta^6 x_3, \zeta x_4,\zeta^3i x_5],$$
\noindent with $\zeta^8=1$.The family $\sG_8$ has dimension 2, see [LZ, Thm.4.15]. The automorphism $\sigma^2$ generates a cyclic group of order 4, hence $X\in \sG_8 \subset \sG_4$.
As such the cubic fourfold $X$ has an equation  as in 2.11  with
\begin{equation} \label{order 8} N_1 = \alpha x_2 x_3 \ , \  N_2=\beta x^2_2 + \gamma x^2_3 \ , \  F(x_0,x_1) = a x^3_0+b x_0 x^2_1$$
$$ L_1 =c x_1 \ , \ L_2 = d x_3  \  ,  L_3 =e x_2,\end{equation}
\noindent see [LZ, Thm. 4.15 (3)].The fixed point set of $\sigma^*$, as an automorphism of $F(X) $, consists of  the 6 lines 
$$ \overline{P_1 P_4} \ , \   \overline{P_1 P_5}  \ , \   \overline{P_2 P_3} \ , \  \overline{P_2 P_4} \ , \  \overline{P_3 P_5}\ , \  \overline{P_4 P_5},$$
\noindent see[Fu, V-3]. Let  $l_1 = \overline{P_2 P_4}$ and $l_2 =\overline{P_3 P_5}$ .  A  line joining a point $P=(a_2,a_4) \in l_1$ and $Q= (b_3,b_5) \in L_2$ is contained in $X$ if
$$b_3  a^2_4 -a_2 b^2_5= 0.$$
This equation defines of rational curve on the surface $l_1 \times l_2 \simeq \P^1 \times \P^1 $ and hence it gives a rational curve $R$ of degree 3 on $F(X)$. To the curve $R$ corresponds a cubic scroll 
$S \subset X$. Therefore $X \in \sC_{12}$. \par
If  in 2.12  $\beta = e =0$ then $X$ contains the 3 planes $\Pi_i$  of equation
$$s_i x_0 +t_i x_1= x_3 = x_4=0,$$
\noindent where  $F(s_i, t_i)=0$, i =1,2,3,  with $F(x_0,x_1) = a x^3_0+b x_0 x^2_1 $. The intersection $S \cap \Pi_i$ consists of all points on the line $r$ joining 
$P = (0,0,1,0,0,0) \in l_1$ and $P_4 =(0,0,0,0,0,1) \in l_2$.  Since $r$ is a line of the ruling of $S$ we have $S \cdot \Pi_i =0$ and therefore $X$ is rational and belongs to 
$\sC_8 \cap \sC_{12}$, see [BBH,3.1.3].\par

\section{Lech pairs associated to rational cubic fourfolds}
 Let $X$ be a cubic fourfold with symplectic automorphism group $G$. Let $S= S_G(X)$ be the coinvariant lattice, i.e.  the orthogonal complement of the invariant  sublattice of $L= H^4(X, \Z)$ under the induced action of  $G$.Then $(G,S)$ is a {\it Lech pair}  (see [LZ ,4.1] ) and 
$$S \subset A(X) _{prim}= H^{2,2}(X) \cap H^4(X,\Z)_{prim}$$
\noindent where  equality holds generically. Except for the case $G =\Z/2$, the rank of $A(X)$ is $\ge  12$, see [LZ,Thm. 1.2]. The transcendental lattice $T(X)$ has a primitive embedding into the K3 
lattice $(E_8)^2 \oplus U^3$. Therefore $X$ has a associated K3 surface and therefore is conjecturally rational. R.Laza and Z.Zheng in [LZ,Thm.1.2]  classify 34 pairs $(G,S)$ of Lech pairs arising from cubic fourfolds .
Let $\sM_{(G,S)}$ be the moduli space parametrizing Lech pairs $(G,S)$, where $G$ is a group of symplectic automorphisms of a a cubic fourfold $X$ and $S =S_G(X)$. Then $\sM_{(G,S)}$ has dimension $20 - \rank(S)$, where $\rank S \le 20$, see [LZ,3.2]. The top dimensional moduli spaces $\sM_{(G,S)}$ will correspond to small groups $G$.The maximal cases, i.e. $\rank S=19,20$ correspond to cubic fourfolds $X$ such that the lattice $A(X)= <S,h^2>$ of algebraic cycles, with $h$ the class of a hyperplane section, has rank 20 or 21. In the first case $\sM_{(G,S)}$ has dimension 1 while in the second case
 the cubic fourfolds are  isolated in moduli, because $\dim \sM_{(G,S)}=0$. All these cubics have  associated K3 surfaces with Picard rank equal to 19 or 20.  A K3 surface $S$ whose Picard rank is either 19 or 20
 admits a Shioda-Inose structure, i.e. there is a Nikulin involution $i$ on $S$ such that the desingularization $Y$ of the quotient surface $S/<i>$ is a Kummer surface, associated to an abelian surface$A$. Therefore there  $S$ is birational  to $Y$.\par
A way to construct cubic fourfolds  $X$ such that $\rank S_G(X) =19$ is to consider a degree six K3 surface $S$, admitting $A_6$ as a symplectic group of automorphisms and defined as follows.
$$ S =V(x^2_1+\cdots + x^2_5+(x_1+\cdots+x_5)^2, x^3_1+\cdots+x^3_5 -(x_1+\cdots+x_5)^3).$$
Then the pencil $X_t$ of cubic fourfolds with a symplectic action of $A_6$ is defined as a hyperplane section of a cubic in $\P^6$
$$X_t  :  x_0 +\cdots+x_5= x^3_0 +\cdots +x^3_5+x_6(x^2_0+\cdots + x^2_5) +tx^3_6 =0,$$
see [LZ,Rk.5.7]. Each fourfold $X_t$ is invariant under the cyclic group $G$ of order 5 generated by a permutation of the coordinates $[x_0,x_1,x_2,x_3,x_4,x_5]$. Therefore  $X_t$ is rational and belongs to $\sC_{42}$.\par
There are six groups of  symplectic automorphisms  $G$  of  a cubic fourfold $X$ such that the coinvariant lattice $S_G(X)$ has rank 20.  The list contains the following cubic fourfolds, see [LZ, Thm.1.8] .  \par

(1) The Fermat cubic $X_F : x^3_0+x^3_1+x^3_2+x^3_3+x^3_4+x^3_5 =0$  which is rational and  is contained in the intersection of all Hassett divisors $\sC_d$, with $d$ as in (*) ;  \par

(2) The Klein cubic fourfold  $X_K$  which is rational and belongs to $\sC_{42}$, see \ref{Klein}.\par

(3)There are two smooth cubic fourfolds $X_1, X_2$ with a symplectic action of the group $G =A_7$.  The first one is the Clebsch cubic fourfold $X_1$  defined by the equation in $\P^6$

$$ x^3_0+x^3_1+x^3_2+x^3_3+x^3_4+x^3_5 +x^3_6 = (x_0+x_1+x_2+x_3+x_4+x_5 +x_6)^3$$

The fourfold $X_1$ contains 357 planes consisting of two $S_7$-orbits. One orbit consists of the 105 Fermat type planes
$$x_i+x_j=x_k+x_l= x_m +x_n =x_p =0,$$
\noindent where (i,j,k,l,m,n,p)) is a permutation of (0,1,2,3,4,5,6),see [DIO,(2.6)]. Since  $X_1$ contains two disjoint planes,  it belongs to $\sC_{14}$ and is rational.\par
The second one is the cubic $X_2$ defined by the equation 
$$x^3_1+x^3_2+x^3_3 +(12/5)x_1x_2x_3+x_1x^2_4+x_2x^2_5+x_3 x^2_6 +(4\sqrt{ 15}/9) x_4x_5x_6 =0,$$
\noindent see [YYZ,(6.14)]  and [Ko,Example 2.1]. We also have
$$G= \Aut (X_2) =\Aut^s(X_2) \simeq A_7.$$
The cubic fourfold $X_2$ is invariant under the order 6 automorphism  $\sigma$ of $\P^5$
$$ \sigma :  [x_1,x_2,x_3,x_4,x_5,x_6] \to  [x_1, \zeta x_2, \zeta^2 x_3, -x_4,  \zeta x_5, -\zeta^2 x_6].$$ 
\noindent with $\zeta$  a primitive cubic root of 1. Therefore  $X_2$ is invariant under the cyclic group $H = <\sigma> \subset G$ of order 6. Let
$$\sigma^2 :  [x_1,x_2,x_3,x_4,x_5,x_6] \to [x_1,\zeta^2 x_2,\zeta x_3, x_4,\zeta^2x_5, \zeta x_6]$$
Then, after a change of coordinates $ [x_1,x_2,x_3,x_4,x_5,x_6]  \to [y_0,y_1,y_2,y_3,y_4,y_5]$,  the automorphism $\sigma^2$ acts as the automorphism in \ref{square}. Therefore  $X_2$ belongs to the family $\sF_2$ of cubic fourfolds with a cyclic group of symplectic automorpohism of order 6. As such $X_2$ is rational and belongs to $\sC_{42}$.\par
(4) The unique  cubic fourfold $X_{12}$ with a symplectic automorphism of order 12, described in \ref{12}, which is rational and belongs to $\sC_{14}$.\par

(5) The unique  cubic fourfold $X_{15}$ with a symplectic automorphism of order 15, described in \ref{15}, which is rational and belongs to $\sC_{14}$.\par

(6) Cubic fourfolds  with a symplectic action  of the Mathieu group $M_{10}$. There are two cubics $X_{10}$ and $X'_{10}$ with symplectic group  isomorphic to $M_{10} $, see [YYZ,Thm.6.15].The group $G = M_{10}$ is maximal in the sense that it is not isomorphic to a proper subgroup of the group of symplectic automorphisms for any smooth cubic fourfold $X$, see [YYZ,Thm.6.12]. Moreover the full group of automorphisms for both $X_{10}$ and $X'_{10}$ coincides with $M_{10}$. The Mathieu group $M_{10} $ has a normal subgroup isomorphic to $A_6$ with index 2. Since $A_6$ contains 6 cyclic subgroups of order 5  the cubics $X_{10}, X'_{10}$ are invariant under a cyclic group of order 5. Therefore these fourfolds are defined , for an appropriate choice of coordinates, by an equation of the form  as in \ref{equation}, see [LZ, Thm.1.2 (5) (b)]. Then $X_{10}, X'_{10} \in F_5$, are rational and belong to $\C_{42}$.\par

 \section {Associated K3 surfaces and natural automorphisms } 
Let $X$ be a cubic fourfold amd $\sA_X$ its Kuznetsov component.The Picard number $\rho(\sA_X)$ equals $\rank \  H^{2,2}(X) -1$,  see [Ou,2.5]. If $S$ is a K3 surface associated to $X$ such that $D^b(S) \simeq \sA_X$ then $\rho(\sA_X)$ coincides with the Picard rank of $S$ and $T_S = T_{\sA_X}$. If $X$ has a symplectic group of automorphisms $G$  then $\rho(\sA_X) \ge \rank \ S_G(X)$. If $G$ is
not isomorphic to the trivial group or to $\Z/2\Z$ then $\rho = \rho(\sA_X) \ge 12$. In this case, thanks to a result of D.Morrison in [Mor,Cor 2.10],  there exists a unique K3 surface $S$  with Picard rank $\rho$, such that $D^b(S) \simeq \sA_X$.\par
Let $X \in \sC_d$, where $d =2(n^2 +n+1)$ with $n \ge 2$. Then there exists a  polarized K3 surface $(S,L)$ of degree $d$ and genus $g$, with with $L^2= d=2g-2$  and a surjective rational map
 \begin{equation} \label{rat} \phi : \sF_d \dashrightarrow \sC_d, \end{equation}
\noindent  where $\sF_d$ is the modiuli space of polarized K3 surfaces of degree $d$, such that  $\phi (S) =\vert X \vert $.There is an isomorphism
 \begin{equation} \label{iso} F(X) \simeq S^{[2]} , \end{equation}
\noindent  where  $S^{[2]}$  is the Hilbert  scheme of  points on $S$, see [Hass, Thm.29]. \par
If $X$ admits a  prime order symplectic automorphism $\tau $ a natural question to ask is when $\tau$  is induced by  the automorphism   $\sigma^{[2]}$ of $F(X)$ , via  the isomorphism in \ref{iso},with $\sigma$
ia prime order symplectic automorphism of the K3 surface $S$.In this case the  fixed locus of the automorphism $\tau$ on $X$ consists of a finite number of isolated points, see [Ped, Lemma 5.7].\par 
The moduli space of polarized K3 surfaces with  a  a symplectic automorphism of order $p =3,5,7$ has dimension $m_p$ with $m_3 =7, m_5=3, m_7=1$, see [GS,Cor.5.1]
The cubic fourfolds $X$ in the families $V_1$ and $V_2$  contain two disjoint planes and hence belong to $\sC_{14} $, while  cubics  $X$ in $V_3, F_5, F_7$ belong to $\sC_{42}$. In all  the cases $X$ belongs to the image of a map as in \ref{rat} and there is an isomorphism $F(X) \simeq S^{[2]}$, with $\phi(S) =X$.
\begin{prop} \label{natural} Let  $V_1, V_2, V_3, F_5, F_7$ be the families of  be cubic fourfolds  $X$  with a symplectic automorphism  $\tau$ of prime order $p =3,5,7$, respectively. The families $V_3$ and $F_7$ contain  a  codimension 1 subfamily  where the automorphism $\tau$ is natural., i.e. it is induced by the symplectic sutomorphism $\sigma^{[2]}$ of $F(X)$. \end{prop}
\begin{proof} For $X \in V_1$ the symplectic automorphism  $\tau$ fixes a cubic surface, for $X \in V_2$ two cubic curves and for $X \in F_5$ a line, see [Fu].Therefore in these  two cases the automorphism 
$\tau^*$ on $F(X)$  cannot be induced by a symplectic automorphism of a K3 surface $S$, via the automorphism in \ref{iso}.\par
In the other cases, i:e.  $X \in V_3$ and $ X \in F_7$ ,the fixed locus of $\sigma$  on $X$ consists of a finite number of isolated points. The dimension of $V_3$ is 8 and the dimension of $F_7$ is 2  while the 
dimension of the moduli  space of  K3 surfaces  with a symplectic automorphism of order 3 is 7 and  that  of the moduli  space of  K3 surfaces  with a symplectic automorphism of order 7 is 1.  Therefore  the subfamilies corresponding  to such surfaces, via  the map in \ref{rat}, have codimension 1.
\end{proof}
\subsection{Equivariant Kuznetsov component} Let $X$ be a cubic fourfold with a group action by a finite group $G$. Then the line bundle $\sO_X(1)$ and the semiorthogonal decomposition
$$D^b(X) =<\sA_X,\sO_X,\sO_X(1),\sO_X(2)>$$
\noindent are preserved under the group action of $G$. Hence we obtain the semiorthogonal decomposition
$$D^b_G(X) =<\sA^G_X,<\sO_X>^G,<\sO_X(1)>^G,<\sO_X(2)>^G>$$
\noindent of the equivariant derived category of $X$. The equivariant Kuznetsov component $\sA^G_X$ is defined as the orthogonal complement of 
$$< \ <\sO_X>^G ,<\sO_X(1)>^G,<\sO_X(2)>^G \ >,$$
\noindent see [FFM].\par
It is natural to ask whether $\sA^G_X$ is equivalent to the derived category  of a smooth variety.\par
 If $X$  belongs to the faniily $V_2$ of cubics with a symplectic automorphism  $\sigma$ of order 3 then the quotient $X /G$, with $G =<\sigma>$, has  a  crepant resolution  $Y \to X/G$, see [XH, Thm. 5.8].Therefore $D^b(Y) \simeq D^b_G(X)$ and there is an isomorphism
$$\sA^G_X \simeq D^b(C \times D)$$
\noindent were $C,D$ are elliptic curves , as in \ref{V_2}.\par
If $X$ has a prime order automorphism $\tau$ which is induced by a symplectic automorphism $\sigma$ on a K3 surface $S$, as in \ref{natural}, then
$$\sA^G_X \simeq D^b_G(S) \simeq D^b(Y),$$
\noindent where $Y \to S/ \sigma$  is a minimal resolution, see [XH,3.1]\par
Here we describe families of cubic fourfolds $X$, with a symplectic automorphism of order 3 , such there exists another smooth cubic fourfold $X'$ and an isomorphism
$$\sA^G_X \simeq \sA_{X'}.$$
\begin{prop}  \label{equivariant} Let $S \in \sF_{6d}$  be a K3 surface with a polarization   $L$ of degree $L^2= 6d =2g -2 $, where $g =n^2 +n +2 $, Assume that $S$ has a  symplectic automorphism $\sigma$ of order 3. Let $X \in \sC_{6d} $ be the image of $S$ under the   map in \ref{rat}.  Assume that $( n^2+n+1)/3 +1= m^2 +m+2$, with $m\ge 2$. Then there is a cubic fourfold $X' \in \sC_{2d} $ such that
$$\sA^G_X\simeq \sA_{X'},$$
\noindent where $G =<\sigma>$. \end{prop}
\begin{proof} Let  $Y $ be the K3 surface which is the  minimal desingularization of the quotient $S/\sigma$. The surface $S/\sigma$ has 6 singularities of type $A_2 $ and $Y$ contains 12 irreducible curves, i.e. 6  disjoint pairs of rational curves meeting in a point. In the diagram
$$ \CD
\tilde S@>{	\alpha}>>  S    \\
@VVV                  @V{\pi}VV  \\
Y@>{\beta}>>S/\sigma    \endCD $$
\noindent the surface $\tilde S$ is the blow-up of $S$ at the six isolated fixed points $P_1,\cdots,P_6$ of $\sigma$. The automorphism $\sigma$ extends to an automorphism  of $\tilde S$ \and $Y = \tilde S/\sigma$.\par
The family $\sS$ of polarized K3 surfaces with a symplectic automorphism of order 3 is the union of countably many components of dimension 7. Similarly the family $\sT$ of  K3 surfaces that are quotients of K3 surfaces with a symplectic automorphism of order 3 is the union of countably many components of dimension 7. The correspondence between K3 surfaces in $\sS$ and in $\sT$ has been described in [GM].
If  $S$ is  a K3 surface  with a polarization  $L$ of degree $L^2= 6d =2g -2 $, where $g =n^2 +n +2 $, then the desingularization $Y$ of  $S/\sigma$ has a polarization $H$ of degree $2d$, see [GM,Thm.5.2].  
Therefore $H^2 = 2d = 2g' -2$, with $g'=( n^2+n+1)/3 +1$. Since  $g'=m^2 +m+2$  the   map $\phi$ in \ref{rat} gives  a diagram
\begin{equation} \label{diagram} \CD
\sF_{6d}@>{\phi}>>  \sC_{6d}    \\
@V{f_*}VV                  @.  \\
 \sF_{2d }@>{\phi}>>\sC_{2d}   \endCD \end{equation}
\noindent where $f_*$ is induced by the quotient map $f : S \to S/\sigma$.  The cubic fourfold $X =\phi(S)$ belongs to $\sC_{6d}$ and $F(X) \simeq S^{[2]}$. 
The order 3 automorphism $\sigma$ induces a symplectic automorphism $\sigma^{[2]}$  and  $\sA_X \simeq D^b(S)$. The image $X' =\phi(Y)$ belongs to $\sC_{2d}$, $ F(X') \simeq Y^{[2]}$  
 and  $\sA_{X'} \simeq D^b(Y)$.\par
Let $G= <\sigma>$ and let $D^b_G(S)$ be the derived category of $G$-equivariant coherent sheaves on $S$. Then
$ D^b(Y) \simeq D^b_G(S),$
\noindent see [XH,Thm.3.1]. Therefore we get
$$\sA^G_X \simeq \sA_{X'}$$
\end{proof}
\begin{ex} (1)  For  $ n=4$ we get  $d=7$ ,  $g= 22$ and $g' = 8$. Let $S \in \sF_{24}$ be k3 surface of genus 22 and degree 42, equipped  with a symplectic automorphisms of order 3. There are two rational cubic fourfolds $X$ and $X'$, with $X \in \sC_{42}$ and $X' \in \sC_{14}$, associated to $S $ and to  the minimal resolution $Y \in \sF_{14}$, respectively, such that
$Y\in \sF_{8}$, respectively, such that
$$\sA^G_X \simeq \sA_{X'}.$$
In this case the map  $ \phi : \sF_{42} \dashrightarrow \sC_{42}$ in \ref{rat} can be described explicitely. Given a polarized K3 surface $(S,L) \in \sF_{42}$  of genus 22, for each point $p \in S$ one considers the rational curve
$$\Delta_p = \{\xi \in S^{[2]} : \supp  \xi =p\}.$$
Under the isomorphism $S^{[2]} \simeq F(X)$ , where $X = \phi(S)$ , the curve $\Delta_p$ corresponds to a degree 9 scroll $R_p \subset X$ with 8 nodes. Conversely, given an element $(X,R)$, with $X  \in \sC_{42} $ and $R$ an eight-nodal curve of degree 9, there exists precisely one element $\phi^{-1}(X)=(S,L,p)$, see [FV, Thm.1.2].\par
(2) For  $n=16$ we get $d=546 $ , $g =274$  and $g'=92$. Therefore if $S$ is a  polarized K3 surface of genus 274  with a symplectic automorphism of order 3  the minimal resolution $Y$ of $S/\sigma$ is a polarized K3 surface of genus 92.  There are two  cubic fourfolds $X \in \sC_{546}$ and 
$X' \in \sC_{182}$ associated to the K3 surfaces  $S$ and  $Y$, respectively, such that
$$\sA^G_{ X} \simeq \sA_{X'}$$
\end{ex}

\end{document}